\documentclass[12pt]{amsart}
\usepackage{graphicx}
\usepackage{amssymb}
\usepackage{amsmath}
\usepackage{epstopdf}
\usepackage{amsthm}
\usepackage{color}
\usepackage{enumerate}
\usepackage{wrapfig}
\usepackage{pifont}
\usepackage{hyperref}

\newtheorem{itheorem}{Theorem} 

\newtheorem{iconjecture}[itheorem]{Conjecture} 
\newtheorem{theorem}{Theorem}[section] 
\newtheorem{proposition}[theorem]{Proposition} 
\newtheorem{lemma}[theorem]{Lemma} 
\newtheorem{corollary}[theorem]{Corollary}

\newtheorem{innerctheorem}{Theorem}
\newenvironment{ctheorem}[1]
  {\renewcommand\theinnerctheorem{#1}\innerctheorem}
  {\endinnerctheorem}
\newtheorem{innercconjecture}{Conjecture}
\newenvironment{cconjecture}[1]
  {\renewcommand\theinnercconjecture{#1}\innercconjecture}
  {\endinnercconjecture}
\newtheorem{innercproposition}{Proposition}
\newenvironment{cproposition}[1]
  {\renewcommand\theinnercproposition{#1}\innercproposition}
  {\endinnercproposition}
\newtheorem{innerccorollary}{Corollary}
\newenvironment{ccorollary}[1]
  {\renewcommand\theinnerccorollary{#1}\innerccorollary}
  {\endinnerccorollary}

\newtheorem{problem}{Problem}

\theoremstyle{remark}

\newtheorem{example}[theorem]{Example}
\theoremstyle{definition}
\newtheorem{definition}[theorem]{Definition} 
\newtheorem{definitions}[theorem]{Definitions} 
\newcommand{\ie}{i.e.,~} 

\def\cC{\mathcal{C}}
\def\cD{\mathcal{D}}

\def\cM{\mathcal{M}}

\def\cS{\mathcal{S}}
\def\cT{\mathcal{T}}

\def\cV{\mathcal{V}}

\def\RR{\mathbb{R}}
\def\R{\mathbb{R}}

\def\ZZ{\mathbb{Z}}
\def\Z{\mathbb{Z}}
\def\codeg{{\rm codeg}}
\def\deg{{\rm deg}}

\newcommand{\Gale}[1]{{#1}^{\star}} 
\newcommand{\bGale}[1]{\bar{#1}^{\star}} 

\newcommand{\veczero}{\boldsymbol{0}}

\def\conv{\mathrm{conv}}

\def\verts{\mathrm{vert}}
\def\aff{\mathrm{aff}\:}
\def\lin{\mathrm{lin}}

\def\rank{\mathrm{rank}\:}

\def\codeg{{\rm codeg}}
\def\deg{{\rm deg}}
\def\degc{{\deg}}
\def\codegc{\codeg}
\def\degG{\Gale{\deg}}
\def\codegG{\Gale{\codeg}}
\def\degZ{{\deg_{\ZZ}}}
\def\codegZ{{\codeg_{\ZZ}}}

\def\degreeG{{$\Gale{\text{degree}}$} }
\def\CayleyG{{Cayley$\Gale{}$} }

\def\dd{{\delta}}
\def\kk{{\kappa}}

\def\Ce{\cC}
\def\Tv{\cD}

\def\S{S}

\DeclareMathOperator{\Dep}{Dep}

\newcommand{\sprod}[2]{\langle {#1} , {#2} \rangle} 
\newcommand{\defn}[1]{\emph{#1}} 
\newcommand{\set}[2]{\ensuremath{\left\{#1\,\middle|\,#2\right\}}} 
\newcommand{\ffloor}[2]{\left\lfloor{\frac{#1}{#2}}\right\rfloor} 

\newcommand{\card}[1]{\vert {#1} \vert}

\newcommand{\ol}[1]{\overline{#1}}
\newcommand{\startproblem}{$\blacktriangleright$\,\,} 

\begin{document}

\title[The degree of point configurations]{The degree of point configurations:\\ Ehrhart theory, Tverberg points and almost neighborly polytopes}

\author{Benjamin Nill}
\address{Department of Mathematics, Stockholm University, SE - 10691 Stockholm, Sweden.}
\email{nill@math.su.se}
\author{Arnau Padrol}
\address{Institut f\" ur Mathematik, Freie Universit\"at Berlin,  Arnimallee 2, 14195 Berlin, Germany.}
\email{arnau.padrol@fu-berlin.de} 

\begin{abstract}
The degree of a point configuration is defined as the maximal codimension of its interior faces. This concept is
motivated from a corresponding Ehrhart-theoretic notion for lattice polytopes and is related to neighborly polytopes and the generalized lower bound theorem and, by Gale duality, to Tverberg theory.

The main results of this paper are a complete classification of point configurations of degree 1, as well as a structure result on point configurations whose degree is less than a third of the dimension. Statements and proofs involve the novel notion of a weak Cayley decomposition, and imply that the $m$-core of a set $\S$ of $n$ points in $\RR^r$ is contained in the set of 
Tverberg points of order $({3m-2(n-r)})$ of $\S$.
\end{abstract}


\maketitle
\section{Introduction and motivation}
\label{intro-sec}

\subsection{Introduction}
Consider the following three problems arising from different contexts.\smallskip
 
\startproblem Let $P$ be a lattice $d$-polytope (a polytope with vertices in the lattice $\Z^d$). The generating function enumerating the number of lattice points in 
multiples of $P$ is of the form:
\[\sum\limits_{k \geq 0} \card{(k P) \cap \Z^d} \, t^k =  \frac{h^*_P(t)}{(1-t)^{d+1}},\]
where the polynomial $h^*_P(t) = \sum_{i=0}^d h^*_i t^i$ is called the {\em $h^*$-polynomial} of~$P$ and its degree $\deg(h^*_P(t))$ is between $0$ and $d$.

\begin{problem}\label{prob:Ehrhart}
Classify the lattice $d$-polytopes whose $h^*$-polynomial's degree is bounded by a fixed constant.
\end{problem}

\startproblem A $d$-dimensional point configuration $A$ is \defn{$k$-almost neighborly}, if every subset of $A$ of size at most $k$ lies in a common face of $\conv(A)$, and it is \defn{$k$-neighborly}, if every subset of $A$ of size $\leq k$ is the vertex set of a face of $\conv(A)$.

A classical result states that if a $d$-dimensional point configuration is $k$-neighborly for any $k>\ffloor{d}{2}$, then it must be the vertex set of a $d$-dimensional simplex. What should be the analogous result for almost neighborly point configurations?

\begin{problem}\label{prob:AlmostNeighborly}
 Find structural constraints for $k$-almost neighborly point configurations when $k$ is small with respect to the dimension.
\end{problem}

\startproblem  Let $\S$ be a configuration of $n$ points in $\RR^r$. A point $x\in \RR ^r$ is a \defn{Tverberg point} of order $m$ (or \defn{$m$-divisible}), if there exist $m$ disjoint subsets $\S_1,\dots,\S_m$ of $\S$ such that $x\in \conv (\S_i)$ for $i=1,\dots m$. The set of Tverberg points of order $m$ of $\S$ is denoted by $\Tv_m(\S)$. Tverberg's Theorem asserts that $\Tv_m(\S)\neq \emptyset$ whenever $n\geq (m-1)(r+1)+1$, a bound that is tight. However, little is known about conditions that can ensure $\Tv_m(\S)\neq \emptyset$ even if $n< (m-1)(r+1)+1$.

A point $x\in \RR ^r$ is in the \emph{$m$-core} of $\S$, denoted by $\Ce_m(\S)$, if every closed halfspace containing $x$ also contains at least $m$ points of $\S$ (\ie $x$ is at \defn{halfspace depth} $m$). It is trivial to see that $\Tv_m(A)\subseteq\Ce_m(A)$, while usually $\Tv_m(A)\nsupseteq\Ce_m(A)$. 
\begin{problem}\label{prob:Tverberg}
What is the largest~$m'$ such that $\Ce_m(A)\subset \Tv_{m'}(A)$?
\end{problem}

While these problems might seem disconnected, they are actually strongly related. We will explain how they are linked and use the intuition of recent results concerning Problem~\ref{prob:Ehrhart} to provide partial answers for Problems~\ref{prob:AlmostNeighborly} and~\ref{prob:Tverberg}. 
We hope that this opens a two-way path between Ehrhart theory and geometric combinatorics, and that future advances on Problems~\ref{prob:AlmostNeighborly} and~\ref{prob:Tverberg} will also be used to improve our knowledge of Problem~\ref{prob:Ehrhart}.

Let us explain the relation between Problem~\ref{prob:Ehrhart} and Problem~\ref{prob:AlmostNeighborly} briefly. Given a lattice $d$-polytope $P$, the degree of $h^*(P)$ is given as $d-k$ where $k$ is the largest positive integer such that $kP$ has no interior lattice points. Now, here is our naive observation: this clearly implies that any set of $k$ lattice points in~$P$ has to lie in a common facet, since otherwise their sum would lie in the interior of $kP$. Therefore, $\ZZ^d\cap P$ is a $k$-almost-neighborly point configuration. Understanding constraints for almost neighborly configurations is a first step for understanding lattice polytopes of bounded Ehrhart $h^*$-degree.

Gale duality provides the translation between Problems~\ref{prob:AlmostNeighborly} and~\ref{prob:Tverberg}. Indeed, $k$-almost neighborly configurations correspond to configurations that contain the origin in their $(k-1)$-core, and vice versa. And as it  turns out, Tverberg points of order~$m$ are in correspondence with so-called weak Cayley decompositions of length~$m$, which is a central concept in our study of almost neighborly configurations.

Section~\ref{intro-sec} of this paper contains the summary of our main results and their interpretations in different contexts, in particular the relation with the problems stated above. The reader is encouraged to skim through it according to background and interest. At the center of our presentation is the notion of the degree of a point configuration. We hope to convince the reader that this is a natural and worthwhile invariant to study. 
In Section~\ref{sec:vectorconfigurations}, we introduce the \degreeG of a vector configuration (its dual counterpart), which is the language used for our proofs. We show the equivalence of the different formulations of our results. Their proofs are contained in Sections~\ref{sec:degk} and~\ref{sec:deg1}.   

\subsubsection*{Acknowledgements} 
The authors want to thank Aaron Dall and Julian Pfeifle for many stimulating conversations, and Alexander Esterov for sharing his ideas and insights. 
AP is supported by the DFG Collaborative Research Center SFB/TR~109 ``Discretization in Geometry and Dynamics'' as well as by AGAUR grant 2009 SGR 1040 and FI-DGR grant from Catalunya's government and the ESF. 
BN is supported by the US National Science Foundation (DMS 1203162).

\subsection{The main notions and results}\label{sec:mainnotions}

Let $A$ be a finite point configuration in $\RR^d$. We will always require that $A$ is full-dimensional (i.e., its affine span equals $\R^d$), and we will allow that $A$ contains repeated points.
We say that a non-empty subset $S\subset A$ is an \defn{interior face} of $A$, if $\conv(S)$ does not lie on the boundary of $\conv (A)$. 
Recall that a {\em facet} of a polytope is a codimension one face. Here are our main definitions:

\begin{definition}
 \label{def:combdegree}
The \defn{degree}, $\degc(A)$, is the maximal codimension of an interior face of~$A$. 
The \defn{codegree} of $A$ is given as $\codegc(A) := d+1-\degc(A)$ 
and equals the maximal positive integer $\kk$ such that every subset of $A$ of size $< \kk$ lies in a common facet of $\conv(A)$.

\end{definition}

In particular, $0\leq \degc(A)\leq d$; and we are interested in those configurations where $\degc(A)\ll d$. Examples of such configurations are $k$-fold pyramids, because $\degc(A)=\degc(A')$ whenever $A'$ is a pyramid over $A$ (see Corollary~\ref{cor:pyramid}). The following converse statement (that takes into account the number of points of the configuration) is proved in Section~\ref{sec:degk}.

\begin{ccorollary}{\ref{cor:easybound}}
Any $d$-dimensional configuration of $n$ points and degree~$\dd$ such that \(d \ge \dd + \frac{n-1}{2}\) is a pyramid.
\end{ccorollary}

For our next example of configurations of small degree, we need the following definition, inspired by the concept of a \defn{Cayley polytope} (see Section~\ref{sec:PrimalCayley}):

\begin{definition}\label{def:strongweakcayley}
A point configuration $A$ admits a \defn{weak Cayley decomposition} of length $m\ge1$, if there exists a partition $A = A_0\uplus A_1 \uplus \cdots \uplus A_m$, such that for any 
$1\leq i\leq m$, $A\setminus A_i$ is the set of points of a proper face of~$\conv (A)$.
The sets $A_1\dots A_m$ are called the \defn{factors} of the decomposition. 

While we allow~$A_0$ to be the empty set, the factors $A_1, \ldots, A_m$ have to be non-empty, because otherwise $\conv(A\setminus A_i)$ would not be a proper face of~$\conv(A)$.
\end{definition}

For example, vertex sets of Lawrence polytopes admit weak Cayley decompositions, because they are Cayley embeddings of zonotopes (see~\cite{HuberRambauSantos2000}). Proposition~\ref{prop:Lawrence} shows that their degree characterizes them. In general, any point configuration that admits a ``long'' weak Cayley decomposition has small degree. 

\begin{cproposition}{\ref{prop:degree-cayley}}
If $A\subset \RR^d$ admits a weak Cayley decomposition of length~$m$, then $\degc(A)\leq d+1-m$.
\end{cproposition}

One of our main results is a converse statement to this proposition:

\begin{itheorem}\label{thm:d+1-3dd}
Any $d$-dimensional point configuration with degree $\dd<\frac{d}{3}$ 
admits a weak Cayley decomposition of length at least $d-3\dd+1$.
\end{itheorem}

For configurations of degree~$\leq 1$, this result can be strengthened. First of all,  Proposition~\ref{prop:deg0} shows that vertex sets of $d$-simplices are the only configurations with $\degc(A)=0$ (up to repeated points). 
This means that the first interesting configurations have $\degc(A)=1$, such as those depicted in Figure~\ref{fig:exdeg1}. In Section~\ref{sec:deg1} we provide a complete classification of these configurations.

\begin{itheorem}\label{thm:deg1}
Let $A$ be a $d$-dimensional configuration of $n$ points. Then $\degc(A)\leq1$ if and only if one of the following holds (up to repeated points)
\begin{enumerate}
\item\label{it:dd=1:d=1} $d\leq 1$; 
or
\item\label{it:dd=1:d=2} $d \ge 2$ and $A$ is a $k$-fold pyramid over a two-dimensional point configuration without interior points; 
or
\item\label{it:dd=1:prism} $d \ge 3$ and $\conv(A)$ is a $k$-fold pyramid over a prism over a simplex with the non-vertex points of $A$ all on the ``vertical'' edges of the prism; or 
\item\label{it:dd=1:simplex} $d \ge 3$ and $\conv(A)$ is a simplex with all non-vertex points of $A$ on the edges adjacent to a vertex $a$ of $\conv(A)$.
\end{enumerate}
\end{itheorem}

\begin{figure}[htpb]
\centering
	\includegraphics[width=.2\textwidth]{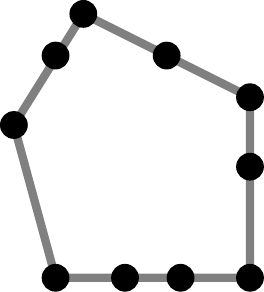}\qquad\qquad\includegraphics[width=.2\textwidth]{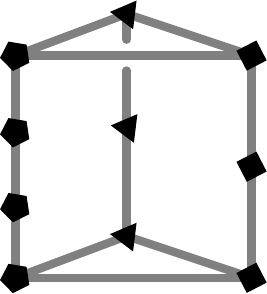}\qquad\qquad\includegraphics[width=.2\textwidth]{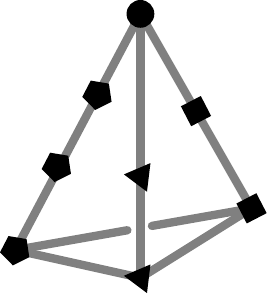}
\caption{Three point configurations of degree $1$ (weak Cayley decompositions indicated with point shapes).}\label{fig:exdeg1}
\end{figure}

The reader may have noticed that this classification implies that if $\degc(A)=1$, then $A$ has a weak Cayley decomposition of length at least $d-1$ (and of length $d$ if $d>2$). This observation (among others, as will be explained below) motivates our main conjecture:

\begin{iconjecture}\label{conj:d+1-2dd}
Any $d$-dimensional point configuration of degree $\dd<\frac{d}{2}$ admits a weak Cayley decomposition of length at least $d+1-2\dd$. 
\end{iconjecture}

The conjectured bound (if correct) is
sharp by Example~\ref{ex:sharpconj}, which shows that the join of $k$ pentagons is a configuration of degree~$\dd=k$ in dimension $d=3k-1$ that does not admit any weak Cayley decomposition of length larger than $d+1-2\dd=k$.

\subsection{Core and Tverberg points}\label{sec:tverberg}

Let $\S$ be a configuration of $n$ points in $\RR^r$. Recall the definitions of $\Ce_m(\S)$ and $\Tv_m(\S)$ from the statement of Problem~\ref{prob:Tverberg}. 
That is, $x\in \Ce_m(\S)$ if every closed halfspace containing $x$ also contains at least $m$ points of $\S$; and $x\in \Tv_m(\S)$ if $x$ belongs to the convex hull of $m$ disjoint subsets of $\cS$. An example is depicted in Figure~\ref{fig:CoreAndDivisible}.

\begin{figure}[htpb]
	\centering
	\begin{tabular}{ccc}
\includegraphics[width=.3\textwidth]{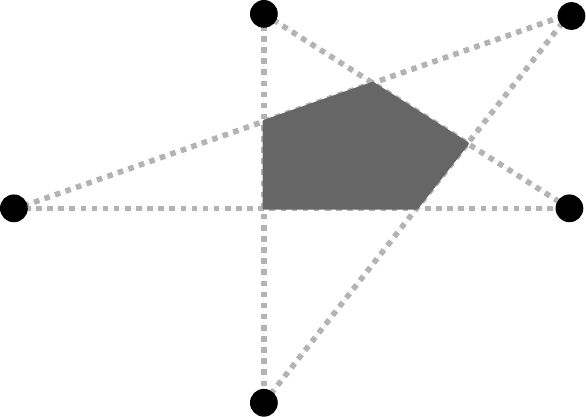}
&\qquad\qquad&
\includegraphics[width=.3\textwidth]{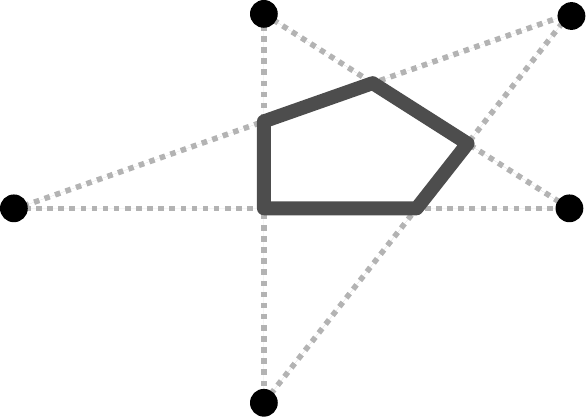}\\
$\Ce_2(\S)$&\qquad\qquad&$\Tv_2(\S)$
	\end{tabular}
\caption{When $\S$ is the vertex set of a pentagon, $\Ce_2(\S)$ is the inner pentagon delimited by the interior diagonals, while $\Tv_2(\S)$ is only the boundary of this inner pentagon.}\label{fig:CoreAndDivisible}
\end{figure}

It is easy to see that $\conv(\Tv_m(\S))\subset \Ce_m(\S)$. Equality was conjectured \cite{Reay1982,Sierksma1982}, and actually holds when $d=2$ or $m=1$. However, Avis found a counterexample for $n=9$, $d=3$ and $m=3$~\cite{Avis1993}, and Onn provided a systematic construction for counterexamples~\cite{Onn2001}.

Tverberg's Theorem asserts that whenever $n\geq (m-1)(r+1)+1$ then $\Tv_m(\S)\neq \emptyset$  (see \cite[Chapter 8]{MatousekLecturesDiscreteGeometry}). In~\cite{Kalai2000}, Kalai asked for conditions that can guarantee that $\Tv_m(\S)\neq \emptyset$ even if $n< (m-1)(r+1)+1$.

As we will explain in Section~\ref{sec:relationtotverberg}, there is a direct correspondence between weak Cayley decompositions and Tverberg points, as well as between the codegree and core points. With it, the proof of Theorem~\ref{thm:d+1-3dd} directly yields the following result, which implies that whenever certain (deep) core points exist,  the set of Tverberg points is also not empty.

\begin{ctheorem}{\ref{thm:d+1-3dd}$_{\text{T}}$}\label{thm:CoreIsDivisible}\label{thm:d+1-3dd_tverberg}
$\Ce_m(\S)\subset \Tv_{3m-2(n-r)}(\S)$.
\end{ctheorem}

Note that this result is only non-trivial if $m > \frac23(n-r)$. On the other hand, $\Ce_m(\S) \not= \emptyset$ implies $m \le n-r$. 
Hence, Theorem~\ref{thm:CoreIsDivisible} is of interest for configurations that admit points in some $m$-core with a relatively large $m$.
In this context our main conjecture, Conjecture~\ref{conj:d+1-2dd}, is equivalent to $\Ce_m(\S)\subset \Tv_{2m-(n-r)}(\S)$.

\subsection{Polytopes, point configurations and triangulations}

\subsubsection{Almost neighborly polytopes}

Recall that a $d$-polytope $P$ is \defn{$k$-neighborly}, if every subset of vertices of $P$ of size $\leq k$ is the set of vertices of a face of $P$. 
The following well-known result (see, for example~\cite[Chapter 7]{Gruenbaum}) motivates the definition of a $d$-polytope as \defn{neighborly}, if it is $\ffloor{d}{2}$-neighborly.

\begin{theorem}\label{thm:kneighd/2}
If a $d$-polytope $P$ is $k$-neighborly for any $k>\ffloor{d}{2}$, then~$P$ must be the $d$-dimensional simplex.
\end{theorem}

Neighborly polytopes are a very important
family of polytopes because of their extremal properties (see~\cite[Sections~9.4]{OrientedMatroids1993}, \cite[Chapter~7]{Gruenbaum}). 
In the definition of $k$-neighborly, one can relax the condition of being the set of vertices of a face by belonging to the set of vertices of a facet.
This concept can be generalized to point configurations (not necessarily in convex position), and gives rise to the definition \defn{$k$-almost neighborly point configuration} as in Problem~\ref{prob:AlmostNeighborly}.
The name `almost neighborly' was coined by Gr\"unbaum in~\cite[Exercices 7.3.5 and 7.3.6]{Gruenbaum}. According to him, this notion was already considered by Motzkin under the name of \emph{$k$-convex sets}~\cite{Motzkin1965}. In \cite{Breen1972} Breen proved that a point configuration is $k$-almost neighborly if and only if all its subconfigurations of size $\leq2d+1$ are. 

In our notation, a configuration $A$ is $k$-almost neighborly if and only if $\codegc(A)>k$. In particular,
Theorem~\ref{thm:deg1} classifies $(d-1)$-almost neighborly point configurations, and Corollary~\ref{cor:easybound} states that any $k$-almost neighborly point configuration with less than $2(k+1)$ points must be a pyramid.
Moreover, Theorem~\ref{thm:d+1-3dd} gives an explicit structure result for $k$-almost neighborly point configurations with $k > \frac{2}{3} d$ in terms of weak Cayley decompositions. Our main conjecture, Conjecture~\ref{conj:d+1-2dd}, would extend this to $k > \frac{d}{2}$. Hence, this can be seen as a potentially precise analogue of Theorem~\ref{thm:kneighd/2} for almost neighborly point configurations.

The concept of almost neighborliness is related to the concept of weakly neighborliness~\cite{Bayer1993}. In particular, in~\cite[Theorem 15]{Bayer1993} Bayer already classified $3$-dimensional polytopes $P$ with $\degc(\verts(P))=1$ as prisms over simplices and pyramids over polygons.

\subsubsection{The Generalized Lower Bound Theorem}
\label{sec:triang}

Let $\cT$ be a $(d-1)$-dimensional simplicial complex, and $f_i(\cT)$ denote the number of $i$-dimensional faces of $\cT$. Then the numbers $h_i(\cT)$ are defined by the polynomial relation \[    \sum_{i=0}^d h_i(\cT)\, t^i= \sum_{i=0}^ d f_{i-1}(\cT)\, t^i \, (1-t)^{d-i}.\]
This polynomial is called the \defn{$h$-polynomial} $h_\cT(t)$ of $\cT$.
 
By the famous $g$-theorem \cite{BilleraLee81,Stanley1980}, $h$-polynomials of the boundary complex of simplicial $d$-polytopes ${P}$ are completely known. In particular, $h_{\partial {P}}(t)$ has degree $d$, it satisfies the Dehn-Sommerville equations $h_i(\partial {P}) = h_{d-i}(\partial {P})$, and it is unimodal (\ie $h_{i}(\partial {P}) \ge h_{i-1}(\partial {P})$ for all $1\le i \le \lfloor d/2\rfloor$). In 1971, McMullen and Walkup~\cite{McMullenWalkup1971} posed the following famous conjecture regarding its unimodality, which is now known as the Generalized Lower Bound Theorem:

\begin{theorem}[Generalized Lower Bound]\label{thm:glbt}
Let ${P}$ be a simplicial $d$-polytope. For $i \in \{1, \ldots, \lfloor d/2\rfloor\}$, 
\begin{enumerate}[(i)]
 \item $h_{i}(\partial {P}) \geq h_{i-1}(\partial {P})$; and
 \item $h_{i}(\partial {P}) = h_{i-1}(\partial {P})$ if and only if ${P}$ can be triangulated without interior faces of dimension $\leq d-i$.
\end{enumerate}
\end{theorem}

The first part of the conjecture was solved by Stanley in 1980, as a part of the proof of the $g$-theorem~\cite{Stanley1980}. The second part of the conjecture had remained open until very recently, when it was proved by Murai and Nevo~\cite{MuraiNevo2012}. 

It is instructive to reformulate the previous theorem. For this, let us consider a triangulation $\cT$ of an arbitrary $d$-polytope ${P}$. 
An \defn{interior face} of $\cT$ is a face of $\cT$ that is not contained in a facet of ${P}$. In this situation, the degree of the $h$-polynomial of $\cT$ is well-known, see \cite[Prop.~2.4]{McM04} or \cite[Corollary~2.6.12]{DeLoeraRambauSantosBOOK}. 

\begin{proposition}
Let $\cT$ be a triangulation of a polytope. Then \linebreak $\degc(h_\cT(t))$ equals the maximal codimension of an interior face of $\cT$.
\end{proposition}

Considering again a simplicial $d$-polytope ${P}$, one defines $g_0(\partial {P}) := 1$, and $g_i(\partial {P}) = h_{i}(\partial {P})-h_{i-1}(\partial {P})$ for $i=1, \ldots, \lfloor \frac{d}{2} \rfloor$. They form the coefficients of the so-called \defn{\text{$g$-polynomial}} $g_{\partial {P}}(t)$. Therefore, Theorem~\ref{thm:glbt} yields for a simplicial polytope ${P}$ that
\[\degc(g_{\partial {P}}(t)) = \min \left\{\degc(h_\cT(t)) \;:\; \cT \text{ triangulation of }{P}\right\}.\]
In other words, the degree $s$ of the $g$-polynomial of a simplicial polytope certifies the existence of {\em some} triangulation that avoids interior faces of dimension $\le d-1-s$. Equivalently, the simplicial polytope ${P}$ is called \defn{$s$-stacked} \cite{McMullenWalkup1971}.

For general polytopes, it is also possible to define (toric) $h$- and $g$-polynomials~\cite{Sta-h-vector}. In this case, by \cite{Sta-local} any rational polytope ${P}$ (conjecturally any polytope) satisfies 
\[\degc(g_{\partial {P}}(t)) \le \min \left\{\degc(h_\cT(t)) \;:\; \cT \text{ triangulation of }{P}\right\}.\]
It is known that simplices are the only polytopes for which $\degc(g_{\partial {P}}(t))=0$. Note that the previous inequality may not be an equality. For instance, a $3$-polytope $P$ which is a prism over a pentagon satisfies $\degc(h_\cT(t)) = 2$ for any triangulation, while $\degc(g_{\partial {P}}(t)) = 1$. In this general situation, it is a hard, open problem to classify all polytopes with $\degc(g_{\partial {P}}(t))=1$ (these polytopes are called \defn{elementary}, see Section~4.3 in \cite{KalaiAspectsPaper}).

To describe how our results fit into this framework, let us consider the degree of the vertex set $\verts({P})$ of a $d$-polytope ${P}$. By observing that any interior simplex $S$ of $\verts({P})$ can be extended to a triangulation that uses~$S$ as a face, we see that $\degc(\verts({P}))$ is the maximal codimension of an interior simplex of some triangulation of~${P}$. In other words,
\[\degc(\verts({P})) = \max \left\{\degc(h_\cT(t)) \;:\; \cT \text{ triangulation of }{P}\right\}.\]
Hence, classifying polytopes of degree $\delta$ is equivalent to studying polytopes where {\em all} triangulations avoid interior faces of dimension $\le d-1-\delta$. This problem is more tractable than the one described above, and Theorem~\ref{thm:deg1}  solves it for $\delta=1$.

Finally, a particular motivation for the study of point configurations of degree $1$ comes from the {Lower Bound Theorem} for balls (see~\cite[Theorem 2.6.1]{DeLoeraRambauSantosBOOK}). It states that if $A$ is a $d$-dimensional configuration of $n$ points, then any triangulation using all the points in $A$ has at least $(n-d)$ full-dimensional simplices, and equality is achieved if and only if every $(d-2)$-face of the triangulation lies on the boundary of $\conv(A)$. 
Hence, $\degc(A)=1$ holds precisely when \emph{all} triangulations using all the points of~$A$ have size~$(n-d)$. This reflects the fact that all triangulations of $A$ are stacked. This interpretation of Theorem~\ref{thm:deg1} is already being used by B\"or\"oczky, Santos and Serra in \cite{BoroczkySantosSerra2013} to derive results in additive combinatorics.

\subsubsection{Totally splittable polytopes}

A \defn{split} of a polytope is a 
subdivision with exactly two maximal cells, which are separated by a \defn{split hyperplane}. A polytope $P$ is called \defn{totally splittable}, if each triangulation of $P$ 
is a common refinement of splits. In~\cite[Theorem 9]{HerrmannJoswig2010}, Herrmann and Joswig establish a complete classification of totally splittable polytopes: simplices, polygons, prisms over simplices, crosspolytopes and a (possible multiple) join of these.

Two splits of $P$ are called \defn{compatible}, if their split hyperplanes do not intersect in the interior of $P$. 
It is easy to see that for a polytope $P$, the degree of its vertex set $\verts(P)$ is at most $1$ if and only if any triangulation of $P$ is a common refinement of compatible splits. 
As a corollary, every polytope of degree~$1$ is totally splittable. In particular, by analyzing each of the cases of Herrmann and Joswig's result one could deduce an independent proof of Theorem~\ref{thm:deg1} for the case that the points in $A$ are in convex position.

\subsection{The relation to Ehrhart theory}\label{sec:Ehrhart}

\subsubsection{The lattice degree of a lattice polytope}

Let us consider the situation where $P \subset \R^d$ is a \defn{lattice polytope}, i.e., its vertices are in the lattice $\Z^d$. As we mentioned before, the \defn{$h^*$-polynomial} is defined by
\[\sum\limits_{k \geq 0} \card{(k P) \cap \Z^d} \, t^k =  \frac{h^*_P(t)}{(1-t)^{d+1}}.\]
Stanley \cite{Sta80,Sta86} showed that the coefficients of $h^*_P$ are non-negative 
integers. Ehrhart theory can be understood as the study of these coefficients.

The degree of $h^*_P(t)$, i.e., the maximal $i \in \{0, \ldots, d\}$ with $h^*_i\not= 0$, is called the \defn{(lattice) degree} $\degZ(P)$ of $P$ 
\cite{BN07}. The \defn{(lattice) codegree} of $P$ is given as $\codegZ(P) := d+1-\degZ(P)$ and equals the minimal positive integer $k$ 
such that $k P$ contains interior lattice points. In recent years these notions and their (algebro-)geometric interpretations 
have been intensively studied \cite{Bat06,BN07,BN08,DiRHNP11,DN10,HNP09,Nil08}. 

It was already noted in \cite[Prop.~1.6]{BN07} that a lattice $d$-polytope $P$ satisfies \begin{equation}
\label{eq:degcleqdegZ}\degc(P\cap\ZZ^d) \le \degZ(P).\end{equation}
If $P$ is normal (i.e., any lattice point in $k P$ is the sum of $k$ lattice points in~$P$), then $\degc(P\cap \Z^d) = \degZ(P)$. 
However, \eqref{eq:degcleqdegZ} is not an equality in general, as the following example in $3$-space shows: $P = \conv(0,e_1,e_2,e_1+e_2+2e_3)$. 
This is a so-called \defn{Reeve simplex} \cite{Ree57}. It satisfies $\degc(P\cap\ZZ^d)=0$, but $\degZ(P) = 2$.

\subsubsection{Cayley decompositions}\label{sec:PrimalCayley}

Our main results in Section~\ref{sec:mainnotions}, are motivated by analogous statements in Ehrhart theory. 
In particular, the notion of a weak Cayley decomposition originates in the widely used construction of Cayley polytopes, which also play a very important role in the study of the degree of lattice polytopes~\cite{BN07,HNP09}.

\begin{definitions}
 Let $A$ be a point configuration in $\RR^d$. We say that $A$ has a
\begin{itemize}
 \item\defn{lattice Cayley decomposition} of length $m$, if $A\subset \ZZ^d$ and there is a lattice projection $\ZZ^d \to \ZZ^{m-1}$ such that $A$ maps onto $\conv(0,e_1, \ldots, e_{m-1})$.
 \item\defn{(combinatorial) Cayley decomposition} of length $m$, 
if there exists a partition $A = A_1 \uplus \cdots \uplus A_m$, such that for any  $\emptyset \not= I \subsetneq \{1, \ldots, m\}$, \(\conv \left( \bigcup\nolimits_{i \in I} A_i\right)\) is a proper face of~$\conv(A)$.

The sets $A_1\dots A_m$ are called the \defn{factors} of the decomposition. Note that they have to be non-empty (because $A_i=\emptyset$ would imply that $\conv(\cup_{i \in \{1, \ldots, m\}\setminus\{i\}} A_i) = \conv(A)$).

\end{itemize}
\end{definitions}

Obviously, if $A$ has a lattice Cayley decomposition, then it has a combinatorial Cayley decomposition whose factors are the preimages of each of the vertices of the simplex. And of course, there are combinatorial Cayley decompositions that are not lattice. However, it is not hard to prove that $A$ has a combinatorial Cayley decomposition of length $m$ if and only if $A$ is combinatorially equivalent (as an oriented matroid) to a configuration $A'$ that can be projected onto the vertex set of a $(m-1)$-simplex.

Despite these analogies, we will see below that the most convenient concept for our purposes turns out to be that of \defn{weak Cayley decompositions}, which we defined in Section~\ref{sec:mainnotions} and that is slightly more general than Cayley decompositions.  Note that the interpretations in Sections~\ref{sec:tverberg} and~\ref{sec:relationtotverberg} also show that it is natural to consider this definition.

Let us remark that the importance of Cayley decompositions arises from the {\em Cayley trick}~\cite{HuberRambauSantos2000,DeLoeraRambauSantosBOOK}. The Cayley trick states that there is a correspondence between configurations $A$ that are a Minkowski sum of $m$ factors, $A=A_1+\dots+ A_m $, with configurations $\widehat A$ that admit a Cayley decomposition of length $m$ (which are known as the \defn{Cayley embedding} of $A_1,\dots,A_n$). With this correspondence there is an isomorphism between the lattice of mixed subdivisions of $A$ and the lattice of subdivisions of $\widehat A$.

\subsubsection{Analogies between the degree and the lattice degree}

From our viewpoint, the degree may be seen as a natural combinatorial generalization of the Ehrhart-theoretic lattice degree.  For example, consider the following properties of the lattice degree of lattice polytopes. Let $P$ be a $d$-dimensional lattice polytope with $r + d + 1$ vertices and lattice degree $\degZ(P)=s$:

\begin{enumerate}[(i)]
\item\label{it:deg0} $P$ has degree $s=0$ if and only if $P$ is unimodularly equivalent to the {\em unimodular simplex} $\conv(0,e_1, \ldots, e_d)$.
\item\label{it:subp} For a lattice polytope $Q \subset P$, we have $\degZ(Q) \le \degZ(P)$ by Stanley's monotoni\-city theorem \cite{Sta93}.
\item\label{it:pyr} If $P$ is a {\em lattice pyramid} over $Q$ (i.e., $P \cong \conv(0,Q\times\{1\}) \subset \R^{d+1}$), then 
$\degZ(P) = \degZ(Q)$.
\item\label{it:deg1} Lattice $d$-polytopes $P$ of degree $s=1$ were classified in \cite{BN07}: 
either $P$ is a $(d-2)$-fold lattice pyramid over the 
triangle with the vertices $(0,0),(2,0),(0,2)$ or $P$ has a lattice Cayley decomposition of length $d$.
\item\label{it:latticepyramid} The following result was shown in \cite{Nil08}: If 
\[d > r(2s + 1) + 4s - 2,\]
then $P$ is a lattice pyramid over an lattice $(d-1)$-polytope.
\item\label{it:latticeCayley} And in~\cite{HNP09}: If $d >f(s) := (s^2+19s-4)/2$, then $P$ has a lattice Cayley decomposition of length $d+1-f(s)$.
\end{enumerate}

Let us compare these results with the combinatorial statements for arbitrary point configurations. Let $A$ be a $d$-dimensional point configuration with $r + d + 1$ points and combinatorial degree $\degc(P)=\dd$:

\begin{enumerate}[(I)]
\item $\degc(A) = 0$ if and only if $A$ is the vertex set of a $d$-simplex (Proposition~\ref{prop:deg0}).
\item For $A' \subset A$, we have $\degc(A') \le \degc(A)$ (Corollary~\ref{cor:degpointdeletioncontraction}).
\item If $A$ is a \defn{pyramid} over $A'$, then $\degc(A) = \degc(A')$ (Corollary~\ref{cor:pyramid}).
\item If $\degc(A) = 1$, then $A$ is a $k$-fold pyramid over a polygon of degree $1$ or admits a weak Cayley decomposition of length~$d$ (Theorem~\ref{thm:deg1}).
\item If
$d \ge r+2\dd$, then $A$ is a pyramid (Corollary~\ref{cor:easybound}).
\item If $d>3\dd$, then $A$ admits a weak Cayley decomposition of length $d+1-3\dd$ (Theorem~\ref{thm:d+1-3dd}).
\end{enumerate}

If a lattice polytope has a lattice Cayley decomposition of length $m$, then its lattice degree is at most $d+1-m$. It was asked in \cite{BN07} whether there might be a converse to this. Above statement \eqref{it:latticeCayley} answered this question affirmatively. The assumption in \eqref{it:latticeCayley} is surely not sharp, it is conjectured that $f(s)=2s$ should suffice, see \cite{DiRHNP11,DN10}. 
Therefore, it seems at first very tempting to also conjecture the analogue statement in the combinatorial setting, at least for vertex sets of polytopes: Namely, 
that for a $d$-dimensional polytope $P$, $d > 2\degc(\verts(P))$ would imply that $\verts(P)$ has a combinatorial Cayley decomposition of length $d+1-2\degc(\verts(P))$. Note that this statement indeed holds for $\degc(\verts(P))=1$ by Theorem~\ref{thm:deg1}. However, rather surprisingly, this guess is wrong as the following example shows.

\begin{example}\label{ex:liftedexceptionalsimplex}
Consider the $(d+1)$-dimensional point configuration \[A=\{0,2 e_1, \dots , 2 e_d, e_1 + e_{d+1},e_1-e_{d+1}, \dots, e_d + e_{d+1},e_d-e_{d+1}\}.\] It is in convex position (i.e., $A$ is the vertex set of $\conv(A)$) and has degree~$2$. However, $A$ does not admit a combinatorial Cayley decomposition of length $>1$.
\end{example}

Even if the point configuration of Example~\ref{ex:liftedexceptionalsimplex} does not admit a combinatorial Cayley decomposition, the subsets $B_i=\{0,2 e_i, e_i + e_{d+1},e_i - e_{d+1}\}$ fulfill all the necessary conditions except for the {disjointness}. Indeed, this point configuration has a weak Cayley decomposition of length~$d$, with factors $A_i=\{2 e_i, e_i + e_{d+1},e_i - e_{d+1}\}$ (and $A_0=\{0\}$). So, 
Example~\ref{ex:liftedexceptionalsimplex} motivates why even for polytopes (instead of more general point configurations) it is necessary to consider weak Cayley decompositions.

Summing up, Theorem~\ref{thm:d+1-3dd} should be seen as the correct combinatorial analogue of the statement \eqref{it:latticeCayley} for lattice polytopes. Moreover, the conjecture that $f(s)=2s$ suffices in the lattice setting matches precisely Conjecture~\ref{conj:d+1-2dd}.

\section{Vector configurations and the dual degree}\label{sec:vectorconfigurations}

The common setting for the proof of the results announced in previous sections will be that of vector configurations. After introducing the necessary notation, we will state our main results in this dual setting in Section~\ref{Sec:mainresultsvectors} and explain the equivalence of all these theorems in Sections~\ref{sec:relationtodegree} and \ref{sec:relationtotverberg}.

\subsection{Notation}

A \defn{vector configuration} $V:=(v_1,\dots,v_n)$ is a finite set of $n$ (possibly repeated) vectors in $\RR^{r}$, which we will assume to be full dimensional (its linear span is the whole $\RR^r$).

Its sets of \defn{linear dependences} $\Dep(V)$ is
$$\Dep(V)=\set{\lambda\in \RR^n}{\sum_{i=1}^n \lambda_i v_i=\veczero}.$$
The sets $C(\lambda):=\set{i}{v_i\in V \text{ and }\lambda_i\neq 0}$ for $\lambda \in\Dep (V)$ are called the \defn{vectors} of $\cM(V)$, the oriented matroid of $V$. The set of vectors of $\cM(V)$ is denoted by $\cV(V)$.
We see the vectors of $\cM(V)$ as signed sets, since each vector $C:=C(\lambda)$ can be decomposed into $C^+=\set{i}{\lambda_i>0}$ and  $C^-=\set{i}{\lambda_i<0}$.
If $C^-=\emptyset$, we say that $C$ is a \defn{positive vector}. The inclusion minimal vectors are called the \defn{circuits} of $\cM(V)$.

Note that we describe a vector $C$ of $\cM(V)$ as a pair $(C^+,C^-)$ of sets of indices of vectors in $V$. This way we can associate vectors of $\cM(V)$ with vectors of related configurations such as $V/v$ or $V\setminus v$ (see below). However, we will often abuse notation and identify $C$, $C^+$ and $C^-$ with the vector subconfigurations $V_C:=\set{v_i\in V}{i\in C}$, $V_{C^+}:=\set{v_i\in V}{i\in C^+}$ and $V_{C^-}:=\set{v_i\in V}{i\in C^-}$ respectively. Hence, we will use $v_i \in C$ and $i \in C$ interchangeably.

In this context, we will say that a subconfiguration $W\subseteq V$ is a positive vector when there is a positive vector $C$ with $V_{C^+}=W$. Observe that $W$ is a positive vector if and only if the origin $\veczero$ is contained in the relative interior of the convex hull of $W$ (seen as points instead of vectors).

An (oriented) linear hyperplane $h$ is defined by a normal vector $v$ and corresponds to the set of points $h:=\set{x}{\sprod{v}{x}=0}$. Its positive (resp. negative) side is the open halfspace $h^+:=\set{x}{\sprod{v}{x}>0}$ (resp. $h^-:=\set{x}{\sprod{v}{x}<0}$). We denote by $\ol h^+=h\cup h^+$ and  $\ol h^-=h\cup h^-$ the corresponding closed halfspaces.

Fix a vector configuration $V$, and let $h_1$ and $h_2$ be linear hyperplanes with normal vectors $v_1$ and $v_2$. We define their composition \defn{$h_1\circ h_2$} (with respect to $V$) as a hyperplane~$h$ with normal vector $v_1+\varepsilon v_2$ for some very small $\varepsilon$ whose value depends on $V$, $h_1$ and $h_2$. Let $v\in V$. If $\varepsilon$ is small enough, then $v\in h$ if and only if $v\in h_1\cap h_2$, and $v\in h^\pm$ if and only if either $v\in h_1^\pm$ or $v\in h_1$ and $v\in h_2^\pm$.

\subsubsection{Deletion and Contraction} Two handy operations on vector configurations $V$ are deletion and contraction (see \cite[Section 6.3(d)]{ZieglerLecturesOnPolytopes}). 
\label{sec:delcontr}

The \defn{deletion} $V\setminus v$ of $v \in V$ is the configuration $V\setminus \{v\}$. The \defn{contraction} $V/v$ of a non-zero vector $v \in V$ is given by projecting $V$ parallel to $v$ onto some linear hyperplane that does not contain $v$ and then deleting $v$. For example, one can use the map $v_i\mapsto \tilde v_i:=v_i -\frac{\sprod{v}{v_i}}{\sprod{v}{v}}v$, and then $V/v=\set{\tilde v_i}{v_i\neq v}$ (see Figure~\ref{fig:exampleContractionDeletion} for an example). The contraction of $\veczero$ is just its deletion. 

In terms of vectors of $\cM(V)$,
\begin{align*}
\cV(V\setminus v)&=\set{(C^+,C^-)}{(C^+,C^-)\in \cV(V),\; v \not\in C^+ \cup C^-},\\
\cV(V\, /\, v)&=\set{(C^+\setminus \{v\},C^-\setminus \{v\})}{(C^+,C^-)\in \cV(V)},
\end{align*}
where the equalities of vectors $C$ of $V$ and vectors $\tilde C$ of $V/v$ in the previous statement should be understood in the sense that their elements have the same corresponding indices. 
\begin{figure}[htpb]
\begin{center}
\begin{tabular}{ccccc}
\includegraphics[width=.2\textwidth]{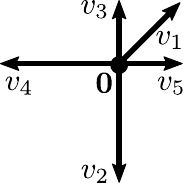}&\qquad\qquad&
\includegraphics[width=.2\textwidth]{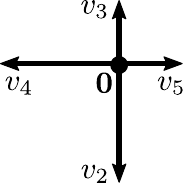}&\qquad\qquad&
\includegraphics[width=.2\linewidth]{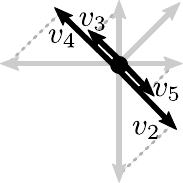}\\\\
$V$&&$V\setminus v_1$&&$V/v_1$
\end{tabular}
\end{center}
\caption{Example of deletion and contraction on~$V$.}
\label{fig:exampleContractionDeletion}
\end{figure}

The definition of deletion and contraction naturally extend to subsets $W\subset V$ by iteratively deleting (resp. contracting)
every element in $W$. In particular, $V/W$ can be obtained by projecting $V\setminus W$ onto a subspace orthogonal to $\lin (W)$. Observe that each linear hyperplane $\tilde h$ in $V/W$ is the image under the projection of a unique hyperplane $h$ in $V$ that goes through the linear span of $W$. 
\begin{lemma}\label{lem:hyperplanescontraction}
If $\pi$ is the projection parallel to $\lin(W)$ onto a complementary subspace, then $\pi$ induces a bijection between hyperplanes in $V$ that contain $\lin(W)$ and hyperplanes in $V/W$, in such a way that $v_i\in h^\pm$ if and only if $\pi(v_i) \in \pi (h)^\pm$.
  \end{lemma}

\subsection{The dual degree}

\begin{definition}
Let $V$ be an $r$-dimensional vector configuration. Its \defn{dual degree} is the nonnegative integer
\begin{equation}\label{eq:defdegG}\degG(V):=\max_{h} |h^+\cap V|-r,\end{equation}
where $h$ runs through all linear hyperplanes of $\RR^r$. 
That is, $\degG(V)=\dd$ if and only if $\dd$ is the minimal integer such that for every linear hyperplane $h$ there are at most $r+\dd$ vectors of $V$ in $h^+$. 

The \defn{dual codegree} of $V$ is defined as \begin{equation}\label{eq:defcodegG}\codegG(V):=\min_{h} |\ol h^+\cap V|,\end{equation}
with $h$ running through all linear hyperplanes. Note that if $|V|=r+d+1$, then \begin{equation}\label{eq:codegG=d+1-degG}\codegG(V)=d+1-\degG(V).\end{equation}

\label{dual-degree-definition-label}
\end{definition}

\begin{example}\label{ex:lawrence}
 Let $V$ be a centrally symmetric configuration of $n$ non-zero vectors in $\RR^r$. Then every linear hyperplane $h$ contains at most one representative of each antipodal pair in $h^+$. Any hyperplane in general position attains this bound, which shows that $\degG(V)=\frac{n}{2}-r$.
\end{example}

A first property of the dual degree of a vector configuration is that it can only decrease when taking subconfigurations and contractions. We omit its easy proof, which follows from the definitions.

\begin{proposition}\label{prop:subconfigurations}
For $v\in V$, $\degG(V\setminus v)\leq \degG(V)$ and $\degG(V/v)\leq \degG(V)$.
\end{proposition}

\subsection{\CayleyG and weak \CayleyG decompositions}
\begin{definitions}
 Let $V$ be a vector configuration in $\RR^r$. Then $V$ admits~a
\begin{itemize}
 \item \defn{(combinatorial) \CayleyG decomposition} of length $m$, 
if there exists a partition  $V=V_1\uplus \cdots \uplus V_m$ such that for $i=1\ldots m$, $V_i$ is a positive vector of~$\cM(V)$. That is, for each factor $V_i$  there is a positive vector ${\lambda^{(i)}}\in \RR^{|V_i|}$ such that $\sum_{v_j\in V_i} ({\lambda^{(i)}})_j\cdot v_j=\veczero$.
 \item \defn{weak \CayleyG decomposition} of length $m$, if it contains $m$ disjoint positive vectors of $\cM(V)$, called the \defn{factors} of the decomposition. 
\end{itemize}
\end{definitions}
Since every positive vector contains a positive circuit, we will often assume that the factors of a weak \CayleyG decomposition are circuits (that is, inclusion-wise minimal).

\begin{proposition}\label{prop:degree-cayley}
If a vector configuration $V$ in $\RR^r$ admits a weak \CayleyG decomposition of length~$m$, then $\degG({V})\leq n-r-m$.
\end{proposition}

\begin{proof}
 If ${V\subset \RR^r}$ has a weak \CayleyG decomposition whose factors are $V_1,\dots,V_m$, then every linear hyperplane ${h}$ contains at least one element of every factor in $\ol {h}^-$. Therefore $|{h}^+\cap {V}|\leq n-m$ for any ${h}$, which proves that $\degG({V})\leq n-r-m$.
\end{proof}

\subsection{The main results for vector configurations}\label{Sec:mainresultsvectors}
Here we restate in terms of vector configurations the main results announced in previous sections. Below we show the equivalence of these theorems, which will be proven in Sections~\ref{sec:degk} and \ref{sec:deg1}.

\begin{ctheorem}{\ref{thm:d+1-3dd}$_{\text{D}}$}\label{thm:d+1-3dd_dual}
 Let $V$ be a vector configuration of rank~$r$ with $r+d+1$ elements and dual degree $\dd:=\degG(V)$. Then $V$ has a weak \CayleyG decomposition of length at least $d-3\dd+1$.
\end{ctheorem}

For the case of configurations of degree~$1$, this result can be improved as follows.
\begin{ctheorem}{\ref{thm:deg1}$_{\text{D}}$}\label{thm:deg1_dual}
 Let ${V}$ be a vector configuration in $\RR^r$ with $n=r+d+1$ elements and $d\ge3$. If $\degG({V})=1$, then ${V}$ has a weak \CayleyG decomposition of length $d$.
\end{ctheorem}

This leads to formulate the following conjecture.
\begin{cconjecture}{\ref{conj:d+1-2dd}$_{\text{D}}$}\label{conj:d+1-2dd_dual}
Any vector configuration $V$ of rank~$r$ and $r+d+1$ elements  and dual degree $\dd:=\degG(V)<\frac{d}{2}$ admits a weak \CayleyG decomposition of length at least $d+1-2\degG(V)$. 
\end{cconjecture}

This conjecture, if true, is easily seen to be sharp.
\begin{example}\label{ex:sharpconj}
 Consider the rank~$r=2$ vector configuration $V$ whose endpoints are the set of vertices of a regular pentagon centered at the origin (this is the Gale dual of a pentagon). This configuration has dual degree~$\dd=1$ and $r+d+1=5$ elements. It admits a weak \CayleyG decomposition of lenght $d+1-2\dd=1$, but it cannot have a weak \CayleyG decomposition of length $2$.

 If we embed $k$ copies of this vector configuration into $k$ orthogonal subspaces of $\RR^{r=2k}$, we obtain a vector configuration of degree~$\dd=k$ with $r+d+1=5k$ elements. It admits trivially a weak \CayleyG decomposition into $d+1-2\dd=k$ factors, and it is not hard to see that it does not admit any decomposition into more factors.
\end{example}

\subsection{The relation with the degree of point configurations}\label{sec:relationtodegree}

The definitions of dual degree and weak \CayleyG decompositions have been chosen in such a way that they correspond to the original definitions of degree and weak Cayley decompositions from Section~\ref{sec:mainnotions}. They are related through \defn{Gale duality}, in the same fashion neighborly point configurations and balanced vector configurations are related (see, for example~\cite{Padrol2012}).

\subsubsection{Gale duality}

We will only provide a very brief summary of some basic results on Gale duality. For an introduction one can consult
~\cite[Lecture 6]{ZieglerLecturesOnPolytopes}, and \cite[Chapter 9]{OrientedMatroids1993} for a more detailed treatment and the relation with oriented matroid duality. 

Gale duality relates a configuration $A:=(a_1,\dots,a_n)$ of $n$ labeled points whose affine span is $\RR^d$ with a configuration $V:=(v_1,\dots,v_n)$ of $n$ labeled vectors in $\RR^{r:=n-d-1}$.  The configuration $V$ is called a \defn{Gale dual} of $A$ and denoted~$\Gale A$. 
Remark that there may be repeated vectors in $V$, even if all the $a_i$ were different.

The key property of Gale duality is that it translates affine evaluations into linear dependencies. We will only need a particular consequence of this statement.

\begin{lemma}\label{lem:gale}
Let $A:=\{a_1,\ldots,a_n\}\subset\RR^d$ as before and $V:=\{v_1,\ldots,v_n\}\subset\RR^{n-d-1}$ denote its Gale dual. For any $I\subset [n]$, let $F:=\set{a_i}{i\in I}$ and $\bGale{F}:=\set{v_i}{i\notin I}$. Then:
\begin{enumerate}[(i)]
 \item\label{it:gale1} $F$ is contained in a supporting hyperplane of $\conv(A)$ if and only if $\bGale{F}$ contains a positive vector of $\cM(V)$.
 \item\label{it:gale2} $F$ are the only points contained in a supporting hyperplane of $A$ if and only if $\bGale{F}$ is a positive vector of $\cM(V)$.
\end{enumerate}
\end{lemma}

With this, we are ready to prove the proposition that translates between the degree and the dual degree.

\begin{proposition}\label{prop:primaldual}
 Let $A:=\{a_1,\ldots,a_n\}\subset\RR^d$ be a point configuration and $V:=\{v_1,\ldots,v_n\}\subset\RR^{n-d-1}$ its Gale dual. Then,
 \begin{itemize}
  \item $\degc(A)=\degG(V)$ and $\codegc(A)=\codegG(V)$, 
  \item $A$ admits a combinatorial (resp. weak) Cayley decomposition with factors $A_1,\dots, A_m$ if and only if $V$ admits a combinatorial (resp. weak) \CayleyG decomposition with factors $V_1,\dots, V_m$, where $V_i:=\set{v_j}{a_j\in A_i}$.
 \end{itemize}
\end{proposition}
\begin{proof}
 We prove first that $\degc(A)=\degG(V)$. By definition, $\degc({A})=\dd$ if and only if every subset $S$ of ${A}$ of size $d-\dd$ is contained in a supporting hyperplane of $\conv(A)$. Equivalently, if $W$ contains the origin in its convex hull for every $W\subset V$ of size $n-d+\dd=r+\dd+1$ (see Lemma~\ref{lem:gale}). Therefore, if~$\degc({A})=\dd$ there cannot be a hyperplane $h$ in $\RR^r$ through the origin that contains more than $r+\dd$ vectors of $V$ in $h^+$ (by the Farkas Lemma, see \cite[Section 1.4]{ZieglerLecturesOnPolytopes}). This proves that $\degG(V)\leq \deg(A)$.
  Conversely, if there is a set of $r+\dd$ vectors whose convex hull does not contain the origin (which by Lemma~\ref{lem:gale} means that there is an interior face of $A$ of cardinality $\leq d+1-\dd$), then we can separate this set from the origin by a hyperplane~$h$, again by the Farkas Lemma. This proves that $\degG(V)\geq \deg(A)$ and hence that $\degG(V)= \deg(A)$. Moreover, by~\eqref{eq:codegG=d+1-degG} \[\codegG(V)=d+1-\degG(V)=d+1-\degc(A)=\codegc(A).\]
 
 Now assume that ${A} = A_0\uplus A_1 \uplus \cdots \uplus A_m$ is a weak Cayley decomposition. That is $A\setminus A_i$ is the set of points in a proper face of~$\conv({A})$, for any $1\leq i\leq m$. Then, by Lemma~\ref{lem:gale}, $V_i=\set{v_j}{a_j\in A_i}$ is a positive vector of $\cM({V})$ for $0\leq i\leq m$, which contains a positive circuit. Thus, ${V}$ has a weak \CayleyG decomposition of length~$m$. The converse is direct.

 The same argument with $A_0,V_0=\emptyset$ shows the equivalence between combinatorial Cayley decompositions of $A$ and combinatorial \CayleyG decompositions of $V$.
\end{proof}

With this we can see how our results are directly related. Indeed, the fact that Theorem~\ref{thm:d+1-3dd_dual} implies Theorem~\ref{thm:d+1-3dd} is straightforward by Proposition~\ref{prop:primaldual}. Conjecture~\ref{conj:d+1-2dd_dual} translates into Conjecture~\ref{conj:d+1-2dd}. Moreover, since the dual of the direct sum is the join (see \cite[Exercise 9.9]{ZieglerLecturesOnPolytopes} for the definition), Example~\ref{ex:sharpconj} shows that the join of $k$ pentagons proves the tightness of the conjecture.

Theorem~\ref{thm:deg1_dual} implies the classification of Theorem~\ref{thm:deg1}, because it shows that in dimension $d\geq 3$ if $A$ has degree~$1$, then either $A$ is a pyramid or it admits a weak Cayley decomposition of length~$d$. Observe that the dimension of each factor of a weak Cayley decomposition of length $d$ cannot be greater than $1$, since all factors are included in a flag of faces of length $d-1$. 
Factors of dimension~$0$ are just apices of pyramids, which can be ignored by Corollary~\ref{cor:pyramid}. Therefore, the only $d$-dimensional configurations that admit weak Cayley decompositions of length~$d$ are (up to repeated points)
\begin{itemize}
\item either $k$-fold pyramids over prisms over simplices with extra  points on the ``vertical'' edges (in which case $A_0=\emptyset$,  and each vertical edge is a factor of a combinatorial  Cayley decomposition of length $d$);
\item or $d$-simplices $\Delta_{d}$ with a vertex $a$ and points on the edges adjacent to $a$
  (here, $A_0 = \{\aff(a)\cap A\}$ and for each edge $e_i$ of $\Delta_{d}$ containing $a$, ${A_i}:=(e_i\cap A)\setminus A_0$ is a factor of a weak Cayley
  decomposition of length~$d$).
\end{itemize}

Hence, to recover the formulation of Theorem~\ref{thm:deg1} presented in the introduction, we only need to observe that a $2$-dimensional point configuration $A$ has degree $\deg(A)\leq 1$ if and only if it does not have interior points.

\subsection{The relation with core and Tverberg points}\label{sec:relationtotverberg}

\begin{proposition}\label{prop:dualtotverberg}
Let $S=\{s_1,\dots,s_n\}$ be a point configuration in $\RR^r$, and consider the vector configuration $V=\{v_1,\dots,v_n\}$ consisting of the set of vectors joining the origin $\veczero$ to the points in $S$. That is, $v_i=\overrightarrow{\veczero,s_i}$. Then,
\begin{itemize}
 \item $\veczero\in \Ce_m(S)$ if and only if $\codegG(V) \geq m$, and
 \item $\veczero\in \Tv_m(S)$ if and only if $V$ admits a weak \CayleyG decomposition of length $m$.
\end{itemize}
\end{proposition}
\begin{proof}
The origin is in the $m$-core of $S$, if every closed halfspace containg it contains at least $m$ points of $S$. Obviously, it is enough to consider those closed halfspaces $\ol h^+$ that contain the origin in their boundary $h$. For those, a point $s_i\in S$ is contained in $\ol h^+$ if and only if the vector $v_i=\overrightarrow{\veczero,s_i}$ belongs to $\ol h^+$, and therefore the claimed equivalence follows from Definition~\ref{dual-degree-definition-label}\eqref{eq:defcodegG}.

To see that $\veczero\in \Tv_m(S)$ if and only if $V$ admits a weak \CayleyG decomposition of length $m$, recall that $V_i\subset V$ is a positive vector of $\cM(V)$ if and only if its set of endpoints $S_i=\set{s_j}{v_j\in V_i}$ contains the origin in the relative interior of its convex hull.
\end{proof}

From this proposition it is direct to deduce that Theorems~\ref{thm:d+1-3dd} and~\ref{thm:d+1-3dd_tverberg} are equivalent.

\section{Weak \CayleyG decompositions}\label{sec:degk}

This section is devoted to the proof of Theorem~\ref{thm:d+1-3dd_dual}.
\subsection{Subconfigurations and quotients}

The following proposition relates the degree of the restriction of a vector configuration to a subspace to the degree of its contraction. It will become a very useful tool for our proofs.
\begin{proposition}\label{prop:subspacequotient}
Let ${V}$ be a vector configuration and let ${W}\subset {V}$ be a subconfiguration of ${V}$ such that $\lin ({W})\cap {V}={W}$.
If we use the notation
\begin{itemize}
 \item $\rank ({V})=r$, $|{V}|=r+d+1$ and $\degG({V})=\dd$;
 \item $\rank ({W})=r_{W}$, $|{W}|=r_{W}+d_{W}+1$ and $\dd_{W}=\degG({W})$ (in $\RR^{r_{W}}$); and
 \item $\rank ({V}/{W})=r_{/{W}}$, $|{V}/{W}|=r_{/{W}}+d_{/{W}}+1$ and $\dd_{/{W}}=\degG({V}/{W})$,
\end{itemize}
then
\begin{align}
 r&=r_{W}+r_{/{W}},\notag\\
 d&=d_{W}+d_{/{W}}+1,\notag\\
 \dd&\geq \dd_{W}+\dd_{/{W}}.\label{eq:decompositioninequality}
 \end{align}
\end{proposition}
\begin{proof}
By construction, $r=r_{W}+r_{/{W}}$. Moreover, counting the number of elements in ${V}$ we get $r+d+1=r_{W}+d_{W}+1+r_{/{W}}+d_{/{W}}+1$, which implies that $d=d_{W}+d_{/{W}}+1$. 

Since the degree of ${W}$ is $\dd_{W}$, there is an oriented hyperplane ${h}_{W}$ of $\lin ({W})$ that contains $r_{W}+\dd_{W}$ elements of ${W}$ in ${h}_{W}^+$. 
Let ${h}_{W}'$ be a hyperplane of $\RR^r$ such that ${h}'_{W}\cap \lin {W}={h}_{W}$. Note that such a hyperplane always exists, for example take the only hyperplane that contains ${h}_{W}$ and the orthogonal complement of $\lin ({W})$. 
Since ${V}/{W}$ has degree $\dd_{/{W}}$, there is an oriented hyperplane ${h}_{/{W}}$ of the quotient $V/W$ that has $r_{/{W}}+\dd_{/{W}}$ elements of~${V/W}$ at ${h}_{/{W}}^+$. 
By Lemma~\ref{lem:hyperplanescontraction}, there is a hyperplane~${h}_{/{W}}'$ of~$\RR^r$ that contains $\lin ({W})$ such that ${h}_{/{W}}'^+\cap V={h}_{/{W}}^+\cap V/W$ (identifying elements of $V/W$ with the corresponding elements of $V$). Then 
\begin{align*}
r+\dd&\geq |({h}_{/{W}}'\circ {h}_{W}')\cap {V}|\\&=|{h}_{/{W}}^+\cap V/W|+|{h}_{{W}}^+\cap W|=r_{/{W}}+\dd_{/{W}}+r_{W}+\dd_{W}.
\end{align*} 
And therefore, $\dd_{{W}}+\dd_{/{W}}\leq \dd$.
\end{proof}

Observe that we took the ``worst'' hyperplane in $\RR^r$ containing $\lin ({W})$ (worst in terms of $|{h}^+\cap V|$), and slightly perturbed it so that it cut $\lin ({W})$ in its worst hyperplane. The proposition states that this perturbed hyperplane cannot be worse than the worst hyperplane that cuts $V$.

\subsection{Some simplifications}
Before continuing to the proof of Theorem~\ref{thm:d+1-3dd_dual}, we will show how it can be reduced to some special cases of vector configurations.

\subsubsection{Totally cyclic configurations}

\begin{definition}
 A vector configuration $V$ of rank~$r$ is \defn{totally cyclic}, if either $r=0$ or $|h^+\cap V| \geq 1$ for every hyperplane $h$.
\end{definition}

Totally cyclic configurations are precisely those that arise as Gale duals of point configurations (see, for example, \cite[Corollary 6.16]{ZieglerLecturesOnPolytopes}).

\begin{lemma}\label{lem:totallycyclicGaledual}
A vector configuration $V$ is the Gale dual of a point configuration (up to rescaling by positive scalars) if and only if it is totally cyclic. 
\end{lemma}

\begin{lemma}\label{lem:totallycyclicsubconf}
Any vector configuration $V$ with $\codegG(V)\geq 1$ contains a totally cyclic subconfiguration $W\subseteq V$ with $\codegG(W)=\codegG(V)$.
\end{lemma}

\begin{proof}
 The proof is by induction on the rank $r$ of $V$, and trivial if $r=0$ or $r=1$. 
 If $V$ is not totally cyclic, there must be a hyperplane ${h}$ with ${h}^-\cap V=\emptyset$, which we can assume to be spanned by vectors in $V$. Let $W=V\cap {h}$, and observe that $\codegG(V)\geq \codegG(W)$. Moreover, $\codegG(V/W)=0$ because ${h}^-\cap V=\emptyset$. Finally, since $\codegG(V)\leq \codegG(W)+\codegG(V/W)$ by Proposition~\ref{prop:subspacequotient}, we see that $\codegG(V)=\codegG(W)$, and the result follows by induction.
\end{proof}

\subsubsection{Irreducible configurations}

\begin{lemma}\label{lem:remove0}
 $\degG(V)=\degG(V\cup\{\veczero\})$ for any vector configuration~$V$.
\end{lemma}
\begin{proof}
 For every linear hyperplane ${h}$, we have ${h}^+\cap V={h}^+\cap (V\cup\{\veczero\})$; hence, $\degG (V)=\degG (V\cup\{\veczero\})$.
\end{proof}

Therefore, adding and removing copies of the origin to a vector configuration does not change its degree, which motivates the following definition.

\begin{definition}
 We say that a vector configuration ${V}$ is \defn{irreducible}, if it does not contain the origin.
\end{definition}

Here is a simple observation about irreducible vector configurations.

\begin{proposition}\label{prop:easybound}
An irreducible vector configuration ${V}\in \RR^r$ of dual degree $\dd$ cannot contain more than $2r+2\dd$ vectors.
\end{proposition}
\begin{proof}
 Take any generic linear hyperplane ${h}$, so that $V \cap H = \emptyset$. By the definition of $\degG$, there are at most $r+\dd$ vectors in~${h}^+$ and in~${h}^-$.
\end{proof}

In terms of Gale duality, if $A'$ is a pyramid over $A$ (i.e., $A = A' \cup \{p\}$ and $p\notin \aff(A')$), then $\Gale {A'}=\Gale A\cup \{\veczero\}$, adding the origin to $\Gale A$ (cf. \cite[Lecture 6]{ZieglerLecturesOnPolytopes}).
Therefore, rephrasing these statements in the primal setting proves two results that we alluded to before:

\begin{corollary}\label{cor:pyramid}
 If $A'$ is a pyramid over $A$, then $\degc{A'}=\degc{A}$.
\end{corollary}

\begin{corollary}\label{cor:easybound}
Any $d$-dimensional configuration $A$ of $n$ points with 
$d \ge \degc(A)+\frac{n-1}{2}$
is a pyramid. 
\end{corollary}

\subsubsection{Pure vector configurations}

The translation of Proposition~\ref{prop:subconfigurations} into the primal setup reads as follows.
\begin{corollary}\label{cor:degpointdeletioncontraction}
For any point configuration $A$ and for any point $a\in {A}$, $\degc({A}\setminus a)\leq \degc({A})$ and $\degc({A}/a)\leq \degc({A})$.
\end{corollary}

Here, the contraction ${A}/a$ is defined analogously as for vector configurations, using the homogeneization $a_i\mapsto \binom{a_i}{1}$ (see \cite[Lecture 6]{ZieglerLecturesOnPolytopes}).
This explains one of the reasons why it is natural to allow configurations that admit repeated points: even if ${A}$ has no repeated points, ${A}/a$ might contain some.
However, it is straightforward to see that deleting repeated points from~${A}$ changes neither the degree nor the property of having a weak Cayley decomposition: 

\begin{lemma}\label{lem:removerepeated}
If the point configuration $A'$ is obtained from $A$ after deleting all repeated points, then 
$\degc({A})=\degc({A}')$. Moreover, ${A}$ admits a (weak) Cayley decomposition of length $m$ if and only if ${A}'$ does. 
\end{lemma}

For this reason, we usually only consider point configurations without repeated points. Being \defn{pure} is the corresponding concept for vector configurations.

\begin{definition}\label{def:pure}
A vector configuration $V\subset\RR^r$ is \defn{pure} if and only if either $r=0$, or for every linear hyperplane~${h}$, $|{h}^+\cap V|\geq 2$ or $|{h}^- \cap V|\geq 2$. 
\end{definition}

The following lemma is the motivation for this definition. We omit its proof, which follows from oriented matroid duality (see for example \cite[Corollary 6.15]{ZieglerLecturesOnPolytopes}).

\begin{lemma}\label{lem:defpure}
 A point configuration $A$ has no repeated points if and only if its Gale dual $V$ is pure.
\end{lemma}

Using that deletion and contraction are dual operations (see for example \cite[Section 6.3(d)]{ZieglerLecturesOnPolytopes}), Lemma~\ref{lem:removerepeated} get translated as follows.
\begin{lemma}\label{lem:purecontraction}
Each totally cyclic vector configuration $V$ contains a pure subconfiguration $W$ with $\degG(V)=\degG(V/W)$ such that $V$ admits a weak \CayleyG decomposition of length $m$ if and only 
$V/W$ does.
\end{lemma}

Actually, it is easy to prove that this lemma also holds when $V$ is not totally cyclic, but we will only need this formulation.
The next lemma also follows directly from the definition.
\begin{lemma}\label{lem:quotientsofpurearepure}
 If $V$ is a pure vector configuration, then $V/v$ is pure for each $v\in V$.
\end{lemma}

A first consequence of Lemma~\ref{lem:defpure} is the characterization of point configurations of degree~$0$.

\smallskip

\begin{lemma}~\label{lem:puredeggeq1}
 If ${V}$ is a pure vector configuration with $\rank(V)\geq 1$, then $\degG({V})\geq 1$.
\end{lemma}
\begin{proof}
Let ${h}$ be a linear hyperplane spanned by some subconfiguration $W \subset {V}$. By Definition~\ref{def:pure}, we can assume that $|{h}^+\cap {V}|\geq 2$. Then the contraction~${V}/W$ is a pure configuration of rank $1$ that satisfies $\degG({V}/W)\geq 1$ because it has a hyperplane  $\tilde h$ with $|{\tilde h}^+\cap {V/W}|\geq 2$ by Lemma~\ref{lem:hyperplanescontraction}. The result now follows from Proposition~\ref{prop:subconfigurations}.
\end{proof}

\begin{proposition}\label{prop:deg0}
The degree of a point configuration ${A}$ is~$0$ if and only if ${A}$ is the set of vertices of a simplex (possibly with repetitions).
\end{proposition}
\begin{proof}
Because of Corollary~\ref{cor:degpointdeletioncontraction} and Lemma~\ref{lem:removerepeated}, it is enough to see that there are no $d$-dimensional point configurations of degree~$0$ with $d+2$ points, none of which are repeated; this follows from Lemma~\ref{lem:puredeggeq1}.
\end{proof}

By taking the Gale dual of the vertex set of a simplex (possibly with repetitions) we get the following result.

\begin{corollary}\label{cor:deg0}
 Any vector configuration $V$ of rank $r$ with $r+d+1$ elements and $\degG(V)=0$ has a weak \CayleyG decomposition of length $d+1$.
\end{corollary}
\begin{proof}
Observe that $V$ has $\degG(V)=0$ if and only if $\codegG(V)=d+1$. By Lemma~\ref{lem:totallycyclicsubconf}, $V$ has a totally cyclic subconfiguration $W$ with $\codegG(W)=d+1$. This subconfiguration has rank $r'$ and  $r'+d'+1$ elements. Since there are at least $r-r'$ elements of $V$ in $V\setminus W$, then $d'+1\leq d+1$. However, by definition $\codegG(W)\leq d'+1$ and hence $d=d'$. This implies that $\degG(W)=0$.

Since $W$ is totally cyclic, its Gale dual is a point configuration $A$ of degree $0$ (by Lemma~\ref{lem:totallycyclicGaledual} and Proposition~\ref{prop:primaldual}). Hence, by Proposition~\ref{prop:deg0}, $A$ is the vertex set of a simplex. Taking Gale duals, this implies that $W$ is a direct sum of positive circuits, i.e., $\RR^{ r'}$ is a direct sum of some subspaces $U_1, \ldots, U_l$ such that ${W}$ is the union of positive circuits $C_1, \ldots, C_l$ with $C_1 \subset U_1, \ldots, C_l \subset U_l$. These circuits form the factors of a weak \CayleyG decomposition of $W$ of length $d'+1=d+1$ which is also a weak \CayleyG decomposition of~$V$.
\end{proof}

\subsection{The proof of Theorem~\ref{thm:d+1-3dd_dual}}

We will use Proposition~\ref{prop:subspacequotient} to prove Theorem~\ref{thm:d+1-3dd_dual}. Recall that in this dual setting our goal is to find many disjoint positive circuits. In the proof we will iteratively find a subconfiguration ${W}$ of ${V}$ of lower rank that has smaller dual degree. 
Eventually we will find a configuration of degree $0$, and Corollary~\ref{cor:deg0} will certify that in this subconfiguration there are already many disjoint positive circuits.

\begin{ctheorem}{\ref{thm:d+1-3dd_dual}}
 Let $V$ be a vector configuration with $r+d+1$ elements and dual degree $\dd:=\degG(V)$. Then $V$ has a weak \CayleyG decomposition of length at least $d-3\dd+1$.
\end{ctheorem}

\begin{proof}

By Lemmas~\ref{lem:totallycyclicsubconf} and~\ref{lem:purecontraction}, we can assume that ${V}$ is totally cyclic and pure. 
The proof will be by induction on $\dd$. The base case is $\dd=0$, which we know to hold because of Corollary~\ref{cor:deg0}.

Let ${h}$ be any hyperplane spanned by elements of ${V}$. Let ${W}={V}\cap {h}$. Then ${V}/{W}$ is 
pure by Lemma~\ref{lem:quotientsofpurearepure} and has rank~$r_{/{W}}=1$,  $d_{/{W}}+2$ elements and degree $\dd_{/{W}}:=\degG({V}/{W})$. By Lemma~\ref{lem:puredeggeq1}, 
\begin{equation}\label{eq:ddgeq1}
 \dd_{/{W}}\geq 1.
\end{equation}
From Proposition~\ref{prop:easybound} we can deduce that $(d_{/{W}} -2\dd_{/{W}})\leq r_{/{W}}-1=0$. Therefore the previous equation~\eqref{eq:ddgeq1} implies that 
\begin{equation}
\label{eq:dmodhleq-1}
 (d_{/{W}} -3\dd_{/{W}})= (d_{/{W}} -2\dd_{/{W}})-\dd_{/{W}}\leq -\dd_{/{W}}\leq -1
\end{equation}

On the other hand, ${W}$ is a vector configuration of rank $r-1$ with $r+d_{{W}}$ elements and degree $\dd_{W}:=\degG({W})$. By Proposition~\ref{prop:subspacequotient}, 
\begin{equation}\label{eq:dwleqd-1}
\dd_{W}\stackrel{\eqref{eq:decompositioninequality}}{\leq} \dd-\dd_{/{W}}\stackrel{\eqref{eq:ddgeq1}}{\leq} \dd -1.                                                                                                                    
\end{equation}
Moreover, again by Proposition~\ref{prop:subspacequotient} and \eqref{eq:dmodhleq-1},
\[
 d_{W}-3\dd_{W}\stackrel{\eqref{eq:decompositioninequality}}{\geq} (d-3\dd)-(d_{/{W}} -3\dd_{/{W}})-1\stackrel{\eqref{eq:dmodhleq-1}}{\geq} d-3\dd.
\]

Since $\dd_{W}\leq \dd-1$ by~\eqref{eq:dwleqd-1}, we can apply induction on ${W}$, which certifies that~$W$ contains at least $d_{W}-3\dd_{W}+1\geq d-3\dd+1$ disjoint positive circuits, and hence so does~${V}$.
\end{proof}

Of course, this theorem is just a first step, since it only proves that there is a subspace that contains many disjoint circuits, but ignores the vectors outside of this subspace, which could form more disjoint circuits. Yet it is already close to the bound of Conjecture~\ref{conj:d+1-2dd_dual}, which would be optimal.
This should be compared with the situation for the original Ehrhart-theoretical counterpart of the conjecture. The currently best result (see statement \eqref{it:latticeCayley} in Section~\ref{sec:Ehrhart}) is not even linear in the lattice degree.

\section{Configurations of degree~$1$}\label{sec:deg1}

The goal of this section is to prove Theorem~\ref{thm:deg1_dual}.

\subsection{Lawrence polytopes}\label{sec:Lawrence}

Lawrence polytopes form a very interesting family of polytopes (cf.~\cite{BayerSturmfels1990}, \cite[Chapter~9]{OrientedMatroids1993}, \cite{Santos2002} or \cite[Lecture~6]{ZieglerLecturesOnPolytopes}). A \defn{Lawrence polytope} is a polytope $P$ such that the Gale dual $V$ of its vertex set is centrally symmetric (after rescaling with positive scalars). That is, maybe after rescaling, $-V=V$ (as a multiset). In Example~\ref{ex:lawrence} we computed their degree.

The following proposition shows that irreducible Lawrence polytopes can be also characterized as having extreme degree.  
Recall that Proposition~\ref{prop:easybound} stated that every irreducible vector configuration of rank $r$, $r+d+1$ elements and degree $\dd$ fulfills $r\geq d+1-2\dd$; Lawrence polytopes are precisely those that attain the equality.    

\begin{proposition}\label{prop:Lawrence}
 An irreducible vector configuration $V$ of rank $r$, $r+d+1$ elements and degree $\dd$ satisfies $r=d+1-2\dd$ if and only if $V$ is centrally symmetric (up to rescaling).
\end{proposition}
\begin{proof}
Example~\ref{ex:lawrence} shows the ``if'' part. 
To prove the converse, we will show that $W:=\lin (v)\cap V$ is centrally symmetric for each $v\in V$.
Let $d_{W}+2$ be the number of elements of $W$ and $\dd_{W}$ its degree. And let $\dd_{/W}$ be the degree of $V/W$, and $(r-1)+d_{/W}+1$ its number of elements. 

By Proposition~\ref{prop:easybound}, $(d_{/ W}-2\dd_{/W})\leq r-2$, and applying Proposition~\ref{prop:subspacequotient} we get that
 $d_{ W}-2\dd_{ W}\geq (d-2\dd)-(d_{/ W}-2\dd_{/W})-1\geq (d-2\dd)-(r-1)=0$.
Moreover, again by Proposition~\ref{prop:easybound}, $d_{W}-2\dd_{ W}\leq 0$. Therefore $d_{ W}=2\dd_{ W}$, and it is easy to check that any configuration of rank~$1$ fulfilling $d_{ W}=2\dd_{ W}$ must be centrally symmetric (again, up to rescaling).
\end{proof}

\subsection{Circuits in configurations of degree one}
 
In order to prove Theorem~\ref{thm:deg1_dual}, we need the following crucial result about circuits in vector configurations of dual degree~$1$. 
It states that in a pure vector configuration of dual degree~$1$ all small circuits are positive (or negative).

\begin{proposition}\label{prop:nosmallcircuits}
 Let ${V}$ be a pure vector configuration of rank~$r$ with $\degG({V})=1$. If $C$ is a circuit of $\cM({V})$ with $|C^+|>0$ and $|C^-|>0$, then $|C|=r+1$. 
\end{proposition}

\begin{proof}
Consider ${W}={V}\cap \lin (C)$. By construction, $\rank(W)=|C|-1$. If $|C^+|>0$ and $|C^-|>0$, there is a hyperplane ${h}$ in $\lin(C)$ with $C\subset {h}^+$. Indeed, by the Farkas Lemma (see \cite[Section 1.4]{ZieglerLecturesOnPolytopes}), if there is no such  hyperplane, then $C$ must be a positive circuit. Therefore $\degG({W})\geq 1$ because $|{h}^+\cap W|\geq |{h}^+\cap C|=\rank(W)+1$. Since ${V}$ is pure, ${V}/{W}$ is also pure by Lemma~\ref{lem:quotientsofpurearepure}. If moreover $|C|\leq r$, then $\rank({V}/{W})\geq 1$, and by Lemma~\ref{lem:puredeggeq1}, $\degG({V}/{W})\geq 1$. Now, Proposition~\ref{prop:subspacequotient} implies that $\degG({V})\geq \degG({W})+\degG({V}/{W})\ge2$, which contradicts the hypothesis that $\degG({V})=1$.
\end{proof}

We deduce some useful corollaries:

\begin{corollary}\label{cor:norepeatedpuredeg1}
If $V$ is a pure vector configuration of rank~$r\ge2$ and $\degG({V})=1$, then it has no repeated vectors except for, perhaps, the zero vector. 
\end{corollary}
\begin{proof}
  By Proposition~\ref{prop:nosmallcircuits}, any circuit with non-empty positive and negative part has size $r+1\geq 3$.
\end{proof}

\begin{corollary}\label{cor:disjointcircuits}
 Let ${V}$ be a pure $r$-dimensional vector configuration with $\degG({V})=1$. If $C\ne D$ are circuits of $\cM({V})$ with $|C\cup D|\leq r+1$, then $C \cap D=\emptyset$.
\end{corollary}
\begin{proof}
Since $C$~and~$D$ are minimal by definition, there must exist $c\in C\setminus D$ and $d\in D\setminus C$. Therefore, $|C|\leq r$ and $|D|\leq r$ and, by Proposition~\ref{prop:nosmallcircuits}, both $C$ and $D$ may be assumed to be positive circuits. 
Suppose there also exists some $p\in C \cap D$. 
Eliminating $p$ on $C$ and $-D$ by oriented matroid circuit elimination (see \cite{OrientedMatroids1993}), we find a circuit $E$ with $c\in E^+$, $d\in E^-$ of size $|E|\leq |C \cup D|-1 \leq r$. This contradicts Proposition~\ref{prop:nosmallcircuits}.
\end{proof}

Another useful  consequence is that the factors of a weak \CayleyG decomposition of a configuration of dual degree~$1$ are its only small circuits.
\begin{lemma}\label{lem:onlysmallcircuits}
Let ${V}$ be a pure vector configuration of rank~$r$ with $r+d+1$ elements, $d\ge2$ and $\degG({V})=1$. If ${V}$ has a weak \CayleyG decomposition of length $d$ with factors $C_1,\dots,C_d$, and $D$ is a circuit of $\cM({V})$ with $|D|\leq r$, then $D=C_i$ for some $1\leq i\leq d$.
\end{lemma}
\begin{proof}
Assume that $D\neq C_i$ for all $1\leq i\leq d$. If there is some $C_i$ with $|C_i\cap D|= |C_i|-1$, then $|C_i\cup D|\leq r+1$ and we get a contradiction to Corollary~\ref{cor:disjointcircuits}.
Otherwise, if $|C_i\cap D|\leq |C_i|-2$ for all $i$ and $|C_j\cap D|\neq \emptyset$ for some $j$, then 
\begin{align*}
  |C_j\cup D|&\leq n-\left|\bigcup_{i\neq j} \left(C_i\setminus D\right)\right|= n-\sum_{i\neq j}  \left| C_i\right|-\left| C_i\setminus D\right|\leq
 n-(d-1)2\\&=r+d+1-2d+2=r-d+3\leq r+1,
\end{align*}
and we again get a contradiction to Corollary~\ref{cor:disjointcircuits}. Hence, $D$ does not intersect any $C_1, \ldots, C_d$. 
By Proposition~\ref{prop:nosmallcircuits}, $D$ can be assumed to be a positive circuit. Therefore, ${V}$ has a weak \CayleyG decomposition of length $d+1$, so Proposition~\ref{prop:degree-cayley} implies that ${V}$ has dual degree $0$, a contradiction.
\end{proof}

In particular, in the situation of the previous lemma any subset ${W}\subset {V}$ with $|{W}|\leq r$ that does not contain any $C_i$ must be linearly independent.\\

Finally, we state another easy consequence of the Farkas Lemma (see \cite[Section 1.4]{ZieglerLecturesOnPolytopes}).
\begin{lemma}\label{lem:poscircuithyperplane}
 Let $C$ be a positive circuit of a vector configuration $V$, and let $h$ be a hyperplane. If $C\not\subset h$, then $|h^+\cap C|\geq 1$ and $|h^-\cap C|\geq 1$.
\end{lemma}

\subsection{The proof of Theorem~\ref{thm:deg1_dual}}

\begin{ctheorem}{\ref{thm:deg1_dual}}
 Let ${V}$ be a vector configuration in $\RR^r$ with $n=r+d+1$ elements and $d\ge3$. If $\degG({V})=1$, then ${V}$ has a weak \CayleyG decomposition of length $d$.
\end{ctheorem}

\begin{proof}
By Lemmas~\ref{lem:totallycyclicsubconf}, \ref{lem:remove0} and~\ref{lem:purecontraction}, we can assume that $V$ is totally cyclic, irreducible, and pure.

We fix $d\ge3$ and use induction on $r$. By Proposition~\ref{prop:easybound}, $r\geq d-1$, and our base case is $r=d-1$. Proposition~\ref{prop:Lawrence} tells us that $r=d-1$ if and only if ${V}$ is centrally symmetric (up to rescaling). Observe that each of the pairs of antipodal vectors forms a circuit, and hence ${V}$ has a \CayleyG decomposition of length~$d$.

If $r>d-1$, $V$ cannot be centrally symmetric by Proposition~\ref{prop:Lawrence}. Hence, there is some $v\in {V}$ such that ${V}\cap \lin (v)$ is not centrally symmetric. Since ${V}$ does not have multiple vectors by Corollary~\ref{cor:norepeatedpuredeg1}, then $\lin(v)\cap V=\{v\}$, a configuration consisting of a single vector. Note that $\degG(\{v\})=0$. 
By Proposition~\ref{prop:subspacequotient} we know that $\degG(V/v)\leq \degG(V)-\degG(\{v\})=1$, and by Lemma~\ref{lem:puredeggeq1} that $\degG(V/v)\ge1$. Combining these inequalities we see that $\degG({V}/v)=1$.
Therefore, ${V}/v$ is a vector configuration of dual degree~$1$ that is pure (Lemma~\ref{lem:quotientsofpurearepure}), and 
has rank~$r-1$ and $(r-1)+d+1$ elements. By the induction hypothesis, ${V}/v$ has therefore a weak \CayleyG decomposition with factors $\tilde C_1, \dots, \tilde C_d$, say.
For convenience, we define $\tilde C_0:=\left({V}/v\right)\setminus\bigcup\nolimits_{i=1}^d \tilde C_i$. 

By counting the number of elements in $|{V}/v|$, we see that 
\begin{equation}\label{eq:countingcircuitelements} \sum_{i=0}^d |\tilde C_i|=|V /v|=|V|-1=r+d.
\end{equation} 
After subtracting $2d$ from both sides, $|\tilde C_0|\ge0$ implies that 
\[\sum_{i=1}^d{\big( |\tilde C_i|-2\big)}\leq r-d;
\]
in particular, $|\tilde C_i|\leq r-1$ for all $1\leq i\leq d$ because $d\ge3$ and $|\tilde C_j|\geq 2$ for all $j$. 

For each $1 \le i \le d$, $\tilde C_i$ is a positive circuit of $\cM({V}/v)$ that expands to a circuit $C_i$ of $\cM({V})$ 
(see Section~\ref{sec:delcontr}). From now on, we consider subsets of~${V}/v$ as subsets of~${V} \setminus v$ by identifying corresponding elements, so that $\tilde C_i = C_i \setminus \{v\}$. 

Since $|C_i|\leq |\tilde C_i|+1\leq r$, Proposition~\ref{prop:nosmallcircuits} shows that either $v\notin C_i$ or $v\in C_i^+$. Hence, $C_i$ is again a positive circuit with either $C_i^+=\tilde C_i^+$ or $C_i^+=\tilde C_i^+\cup \{v\}$. We will show that if some $C_i$ contains $v$, no other $C_j$ can. This will prove our claim because then $C_1,\dots,C_i,\dots,C_d$ are disjoint positive circuits that form a weak \CayleyG decomposition of~${V}$.

For this, we assume that $v\in C_1\cap C_2$ and reach a contradiction. We start with some definitions.
For $1\leq i\leq d$, let ${D}_i$ be a subset of $|\tilde C_i|-2$ elements of~$\tilde C_i$, and set ${D}:=\tilde C_0 \cup \bigcup\nolimits_{i=1}^d {D}_i$. Next, choose $v_1\in \tilde C_1 \setminus {D}_1$ and $v_2\in \tilde C_2 \setminus {D}_2$ (so that, in particular, $v\notin\{v_1,v_2\}$) and define ${D}':={D}\cup\{v_1,v_2\}$. 

A first observation is that the elements in ${D}'$ must be linearly independent. Indeed, since 
\begin{eqnarray*}
|{D}'|
&=&
2+|\tilde C_0|+\sum_{i=1}^d \big(|\tilde C_i|-2\big)
\\ &\stackrel{\eqref{eq:countingcircuitelements}}{=}&
2 + (r+d) - 2d 
\ = \ 
r+2-d
\ \leq \ 
r-1,
\end{eqnarray*}
already their projections to ${V}/v$ are linearly independent. The reason for this is that if the elements in $D'/v$ were not linearly independent, then they would contain a circuit. But this contradicts Lemma~\ref{lem:onlysmallcircuits} because $D'\not\supseteq \tilde C_i$ for all $i$, since by construction $|\tilde C_i \setminus D'|\ge1$ for all~$i$. Now, let ${h}'$ be a hyperplane through $\lin({D}')$ that is otherwise in general position with respect to~${V}$. This is possible because the rank of~$V$ is~$r$, and $D'$~has at most $r-1$ elements. Observe that $v\notin\lin({D}')$, because otherwise the vectors in ${D}'$ would form a circuit in ${V}/v$. Therefore, $v\notin {h}'$, and we can orient~${h}'$ so that $v\in {{h}'}^-$. Then $|{{h}'}^+\cap \tilde C_i|=|{{h}'}^+\cap C_i|=1$ for $i=1,2$ because of Lemma~\ref{lem:poscircuithyperplane} and our assumption that $v\in C_1\cap C_2$.
 Moreover, since the elements in ${D}'$ are linearly independent, we can perturb~${h}'$ to a hyperplane~${h}$ through $\lin({D})$ such that $v_1,v_2\in {h}^+$. This yields 
\[
\big|{{h}}^+\cap \tilde C_i\big|
=
\big|({{h}'}^+\cap \tilde C_i)\cup v_i\big|
= 2
\qquad\text{for } i=1,2.
\]

Furthermore, we claim that $|{h}^+\cap \tilde C_j|\geq 1$ for all $j\ge3$. 
If, on the contrary, there existed some $j\ge3$ with $|{h}^+\cap \tilde C_j|=0$,  Lemma~\ref{lem:poscircuithyperplane} would yield~$v\notin C_j$ (\ie $C_j=\tilde C_j$), and moreover $C_j$ would be completely contained in~${h}$. Hence, by construction, $C_j$ would be completely contained in $\lin({D})$. In particular, some $v_j\in \tilde C_j \setminus {D}_j$ would satisfy $v_j\notin {D}$ but $v_j\in \lin ({D})$. Therefore, this element would be part of a circuit in~$\{v_j\}\cup {D}$, distinct from~$C_j$ since $|C_j\cap {D}|=C_j-2$. However, $|C_j\cup {D}|\leq |{D}|+3\leq r$, which would contradict Corollary~\ref{cor:disjointcircuits}.

Finally, let ${h}''$ be any hyperplane such that ${D}\subset {{h}''}^+$. Now
\begin{itemize}
\item $\big|({h}\circ {h}'')^+ \cap \tilde C_0\big|=|\tilde C_0|$;
\item $\big|({h}\circ {h}'')^+ \cap \tilde C_i\big|=|\tilde C_i|$ for $i=1,2$; and
\item $\big|({h}\circ {h}'')^+ \cap \tilde C_j\big|\geq |\tilde C_j|-1$ for $3\le j\leq d$. 
\end{itemize}

Therefore, using \eqref{eq:countingcircuitelements} we see that
\[
 \big|({h}\circ {h}'')^+ \cap {V}\big|\ \geq \  \sum_{i=0}^d |\tilde C_i|-(d-2)=r+2,
\]
which contradicts $\degG({V})=1$.

\end{proof} 

\bibliographystyle{plain}
\bibliography{degree}

\end{document}